\documentclass[12pt]{amsart}
\usepackage{amssymb,amsfonts,amsmath,amsthm,cite,verbatim}
\usepackage{xcolor}
\usepackage{graphicx}
\usepackage{array}
\usepackage{booktabs}
\usepackage[linesnumbered,ruled,vlined]{algorithm2e}
\usepackage[left=2.5cm, right=2.5cm, top=3.5cm, bottom=3.5cm]{geometry}
\usepackage{hyperref}
\usepackage[normalem]{ulem}

\allowdisplaybreaks[4]
\newtheorem{theorem}{Theorem}[section]
\newtheorem{thm}[theorem]{Theorem}
\newtheorem{prop}[theorem]{Proposition}
\newtheorem{lem}[theorem]{Lemma}

\newtheorem{conj}[theorem]{Conjecture}

\makeatletter \@addtoreset{equation}{section}

\SetCommentSty{mycommfont}
\SetKwInput{KwInput}{Input}                
\SetKwInput{KwOutput}{Output}

\def\CT{\mathop{\mathrm{CT}}\limits}

\def\H{\mathcal{H}}
\def\m{\boldsymbol{m}}
\def\RR{\mathbb{R}}
\def\TC{\widetilde{C}}

\def\x{\boldsymbol{x}}

\title{A proof of Xin-Zhang's tridiagonal determinant conjecture (extended version)}

\author[Jiaqiang Hu and Chen Zhang]{Jiaqiang Hu$^1$ and Chen Zhang$^{2,*}$}

\address{$^{1}$ School of Mathematical Sciences, LPMC, Nankai University, Tianjin 300071, PR China \\
$^{2}$ Center for Combinatorics, LPMC, Nankai University, Tianjin 300071, PR China}

\email{$^1$\texttt{2210472@mail.nankai.edu.cn}\ \& $^2$\texttt{ch\_enz@163.com}}

\date{\today}

\thanks{$*$ Corresponding author.}

\begin{document}

\begin{abstract}
We confirm a recent conjecture of Xin and Zhang, which establishes a simple product formula for the characteristic polynomial of an $(n-1) \times (n-1)$ tridiagonal matrix $C$. This characteristic polynomial arises from a recurrence relation that enumerates $n \times n$ nonnegative integer matrices with all row and column sums equal to $t$, also called the Ehrhart polynomial of the $n$th Birkhoff polytope. Moreover, we extend our method to broader families of tridiagonal matrices. 
\end{abstract}
\maketitle

\noindent
\begin{small}
\end{small}

\noindent
\begin{small}
\end{small}

\section{Introduction}
Our main objective in this paper is to prove the following conjecture of Xin and Zhang \cite{XZ24+}, which gives a product formula for the characteristic polynomial of certain tridiagonal matrix.
\begin{conj}[{\cite[Conjecture 4.8]{XZ24+}}]\label{conj-detC}
Let $n$ be a positive integer, and let $C$ be the $(n-1) \times (n-1)$ matrix defined by
\[
C_{i,j} = \begin{cases}
            t - (n-2i + 2)(n-1-i) - 1, & \text{if } i = j; \\
            (i-1)(n-i), & \text{if } i = j + 1; \\
            - (n-1-i)(n-i), & \text{if } i = j - 1; \\
            0, & \text{otherwise}.
          \end{cases}
\]
Then
\[
\det C = \begin{cases}
           (t - n + 1) \prod_{i=0}^{r-2} \big( t - n + 1 - r(r-1) + i(i+1) \big)^2, & \text{if } n = 2r; \\
           (t - n + 1)\big(t - n + 1 - (r-1)^2\big) \prod_{i=1}^{r-2} \big( t - n + 1 - (r - 1)^2 + i^2 \big)^2, & \text{if } n = 2r - 1.
         \end{cases}
\]
\end{conj}

This conjecture arises in the study of the Ehrhart polynomial $H_n(t)$ of the $n$th Birkhoff polytope, where $H_n(t)$ counts the number of $n \times n$ nonnegative integer matrices with all row and column sums equal to $t$. The computation of $H_n(t)$ has been a subject of extensive research; see, for example, \cite{AHZ25, CR99, DLY09, MacM60, Mou00, Sta73, Sta76, XZZZ25}. So far, $H_n(t)$ is known only for $n \leq 9$, with $H_9(t)$ computed by Beck and Pixton \cite{BP03} using residue techniques.

Following \cite[Corollary 4]{BP03}, it is known that
\[
H_n(t) = \sum_{\m} \binom{n}{\m} \CT_{\x} \mathcal{H}^{\m}(t),
\]
where the sum is over all nonnegative integer sequences $\m = (m_1, m_2, \dots, m_n)$ such that $\sum_{i=1}^n m_i = n$ and $\sum_{i=1}^r m_i > r$ for all $1 \le r < n$. Here, $\CT_{\x} f$ denotes the constant term of the Laurent series $f$ in $\x = (x_1, x_2, \dots, x_n)$, and
\[
\mathcal{H}^{\m}(t) := \prod_{i=1}^{n} \frac{x_i^{(m_i-1)t}}{\prod_{j=1,j \neq i}^n (1 - x_j/x_i)^{m_i}}.
\]

In particular, for $\m = (2, 1^{n-2})$, we define
\begin{align*}
h_n(t) &:= (-1)^{\frac{(n-1)(n-2)}{2}} \CT_{\x} \mathcal{H}^{(2, 1^{n-2}, 0)}(t) \\
&= \CT_{\x} \frac{(x_n/x_1)^{-t} \prod_{i=2}^{n - 1} (x_i/x_1)^{2 i - n} \prod_{i=1}^{n-1} (1 - x_n/x_i)}{\prod_{i=2}^{n} (1 - x_i/x_1)^3 \prod_{2 \le i < j \le n} (1 - x_j/x_i)^2}.
\end{align*}

Noting the similarity with the Morris constant term identity studied in \cite{BV01} and \cite{XinPhD04}, Xin and Zhang \cite{XZ24+} derived a recurrence for $h_n(t)$ of the form
\begin{equation}\label{e-hrec}
h_n(t) = - \frac{1}{c_0(t+n-1)}\sum_{i=1}^{n-1} c_i(t+n-1) h_n(t-i),
\end{equation}
in which $c_0(t) = \det C$. Therefore, if Conjecture~\ref{conj-detC} holds, then \eqref{e-hrec} indeed yields a recurrence when $t > \big\lfloor \frac{(n-1)^2}{4} \big\rfloor = \begin{cases}
                                r(r-1), & \mbox{if } n = 2 r; \\
                                (r-1)^2, & \mbox{if } n = 2 r - 1.
                              \end{cases}$

Indeed, the ideas in \cite{XZ24+} also apply to $\CT_{\x} \H^{(k, 1^{n-k}, 0^{k-1})}(t)$ for $3 \le k \le n-1$. Moreover, the determinant $\det C$ arises in a similar but more complicated recurrence. These aspects will be elaborated in an upcoming paper \cite{Z+}.

Further developments motivate us to consider the equivalent determinant of 
\[
tI-\TC^{(n-1)} = C + (n - 1) I,
\]
where $I$ is the identity matrix, and $\TC^{(n)}$ is an $n \times n$ tridiagonal matrix with entries
\begin{equation}\label{e-TC}
(\TC^{(n)})_{i,j} = \begin{cases}
            (n+1-i)(n+1-2i), & \text{if } i = j; \\
            (n+1-i)(1-i), & \text{if } i = j + 1; \\
            (n+1-i)(n-i), & \text{if } i = j - 1; \\
            0, & \text{otherwise}.
          \end{cases}
\end{equation}

It is still hard to determine the characteristic polynomial of $\TC^{(n)}$: none of the minors of $tI-\TC^{(n)}$ admit a simple factorization. This prevents us from proving the conjecture by traditional techniques such as diagonalization, LU decompositions, continued fractions, etc.. Ultimately, we discovered that $\TC^{(n)}$ is similar to a lower triangular matrix. We state this result as follows.
\begin{thm}\label{thm-main}
For a positive integer $n$, let $\TC^{(n)}$ be the matrix defined in \eqref{e-TC}. There exists an $n \times n$ nonsingular upper triangular matrix $U$ with entries
\begin{equation}\label{e-U}
U_{i,j} = \binom{n-i}{n-j},
\end{equation}
such that $U \TC^{(n)} U^{-1}$ is a lower triangular matrix whose entries are given by
\[
(U \TC^{(n)} U^{-1})_{i,j} = \begin{cases}
            i(n-i), & \text{if } i = j; \\
            (n+1-i)(1-i), & \text{if } i = j + 1; \\
            0, & \text{otherwise}.
          \end{cases}
\]
Consequently, Conjecture \ref{conj-detC} holds and we have the nicer formula
\[
\det C = \prod_{i=1}^{n-1} (t - n + 1 - i (n-1-i)).
\]
\end{thm}

For instance,
\[
\TC^{(4)} = \left(\begin{array}{cccc}
  12 & 12 & 0 & 0  \\
  -3 & 3 & 6 & 0  \\
  0 & -4 & - 2 & 2  \\
  0 & 0 & -3 & - 3  \\
\end{array}\right),
U = \left(\begin{array}{cccc}
  1 & 3 & 3 & 1  \\
  0 & 1 & 2 & 1 \\
  0 & 0 & 1 & 1  \\
  0 & 0 & 0 & 1  \\
\end{array}\right),
U \TC^{(4)} U^{-1} = \left(\begin{array}{cccc}
  3 & 0 & 0 & 0  \\
  -3 & 4 & 0 & 0  \\
  0 & -4 & 3 & 0  \\
  0 & 0 & -3 & 0  \\
\end{array}\right).
\]

The rest of this paper is Section \ref{sec-pfmain}, in which we prove Theorem \ref{thm-main} using an explicit formula of $U^{-1}$, together with several combinatorial identities. In Section \ref{sec-gen}, we extend these ideas to prove a natural generalization of Theorem \ref{thm-main}.

\section{The proof of Theorem \ref{thm-main}}\label{sec-pfmain}
Our proof of Theorem \ref{thm-main} is based on the following well-known combinatorial identities (the proofs are omitted):
\begin{align}
  \binom{n+1}{k} &= \binom{n}{k} + \binom{n}{k-1}, \label{e-comid1} \\
  n \binom{n-1}{k-1} &= k \binom{n}{k} = (n-k+1) \binom{n}{k-1}, \label{e-comid2} \\
  \binom{n}{k} \binom{k}{j} &= \binom{n}{j} \binom{n-j}{k-j} = \binom{n}{n-k+j} \binom{n-k+j}{j},  \label{e-comid3}
\end{align}
where $n \ge k \ge j$ are nonnegative integers.

\begin{proof}[Proof of Theorem \ref{thm-main}]
By \cite[Theorem 1]{CV93}, the entries of $U^{-1}$ are given by
\[
(U^{-1})_{i,j} =(-1)^{i+j} U_{i,j}= (-1)^{i+j} \binom{n-i}{n-j}.
\]
Consider the $(i,j)$-entry of $U \TC^{(n)} U^{-1}$:
\begin{align*}
(U \TC^{(n)} U^{-1})_{i,j} &= \sum_{1 \le l,m \le n} U_{i,l} (\TC^{(n)})_{l,m} (U^{-1})_{m,j}  \\
&= \sum_{l=1}^n  U_{i,l} (\TC^{(n)})_{l,l} (U^{-1})_{l,j} + \sum_{l=2}^n  U_{i,l} (\TC^{(n)})_{l,l-1} (U^{-1})_{l-1,j} \\
& \quad + \sum_{l=1}^{n-1} U_{i,l} (\TC^{(n)})_{l,l+1} (U^{-1})_{l+1,j},
\end{align*}
where the second equality holds by $(\TC^{(n)})_{l,m} = 0$ when $|l-m|>1$. Observe that $(\TC^{(n)})_{l,l} = (\TC^{(n)})_{l,l-1} + (\TC^{(n)})_{l,l+1}$, and the formulas of $(\TC^{(n)})_{l,l-1}$ and $(\TC^{(n)})_{l,l+1}$ vanish if setting $l=1$ and $l=n$, respectively. The computation is naturally divided into two parts, denoted by $S_1$ and $S_2$, respectively. In detail,
\begin{align*}
S_1 &= \sum_{l=1}^n  U_{i,l} (\TC^{(n)})_{l,l-1} (U^{-1})_{l,j} + \sum_{l=2}^n  U_{i,l} (\TC^{(n)})_{l,l-1} (U^{-1})_{l-1,j} \\
&= \sum_{l=1}^n  U_{i,l} (\TC^{(n)})_{l,l-1} ((U^{-1})_{l,j} + (U^{-1})_{l-1,j}) \\
&= \sum_{l=1}^n  U_{i,l} (\TC^{(n)})_{l,l-1} (-1)^{j+l-1} \Big( \binom{n-l+1}{n-j} - \binom{n-l}{n-j} \Big) \\
&= \sum_{l=1}^n  U_{i,l} (\TC^{(n)})_{l,l-1} (-1)^{j+l-1} \binom{n-l}{n-j-1}. \qquad (\text{by \eqref{e-comid1}}) 
\end{align*}
Similarly, we have
\begin{align*}
S_2 &= \sum_{l=1}^n  U_{i,l} (\TC^{(n)})_{l,l+1} (U^{-1})_{l,j} + \sum_{l=1}^{n-1} U_{i,l} (\TC^{(n)})_{l,l+1} (U^{-1})_{l+1,j} \\
&= \sum_{l=1}^n  U_{i,l} (\TC^{(n)})_{l,l+1} (-1)^{j+l} \binom{n-l-1}{n-j-1}. 
\end{align*}
Using Equation \eqref{e-comid2}, we can further compute $S_2$ as:
\begin{align*}
S_2 &= \sum_{l=1}^n  U_{i,l} (-1)^{j+l} (n+1-l) (n-l) \binom{n-l-1}{n-j-1} \\
&= \sum_{l=1}^n  U_{i,l} (-1)^{j+l} (n+1-l) (j-l+1) \binom{n-l}{n-j-1}.
\end{align*}
Then,
\begin{align*}
(U \TC^{(n)} U^{-1})_{i,j} = S_1 + S_2 
&= j \sum_{l=1}^n (-1)^{j+l} \binom{n-i}{n-l} \binom{n-l}{n-j-1} (n+1-l) \\
&= j \sum_{l=1}^{n} (-1)^{j+l} \binom{n-i}{n-j-1} \binom{j+1-i}{j+1-l} (n+1-l). \qquad (\text{by \eqref{e-comid3}})
\end{align*}
It is clear that the above formula equals $0$ when $i > j + 1$ and equals $(1-i)(n+1-i)$ when $i = j + 1$. When $i \le j$, we can further write
\begin{align*}
(U \TC^{(n)} U^{-1})_{i,j} &= j \sum_{l=i}^{j+1} (-1)^{j+l} \binom{n-i}{n-j-1} \binom{j+1-i}{j+1-l} ((n-j) + (j+1-l)) \\
&= - j (n-j) \binom{n-i}{n-j-1} \sum_{l=i}^{j+1} \binom{j+1-i}{j+1-l} (-1)^{j+1-l} \\
& \quad + j (j+1-i) \binom{n-i}{n-j-1} \sum_{l=i}^{j} \binom{j-i}{j-l} (-1)^{j-l} \qquad  (\text{by \eqref{e-comid2}}) \\
&= \binom{n-i}{n-j-1} (- j (n-j) (1+(-1))^{j+1-i} + j (j+1-i) (1+(-1))^{j-i}) \\
&= \begin{cases}
     i(n-i), & \mbox{if } i=j; \\
     0, & \mbox{if } i<j.
   \end{cases}
\end{align*}
This completes the proof.
\end{proof}

As a byproduct, we have the following result.

\begin{prop}\label{prop-lowpart}
Let $U$ be the matrix with entries as in \eqref{e-U}, and $M$ be a matrix satisfying $M_{i,j} =0$ for all $i-j \ge 2$. Then we have $(UMU^{-1})_{i,j} = M_{i,j}$ for all $i - j \ge 1$.
\end{prop}
Proposition \ref{prop-lowpart} can be visualized as
\[
U \left(\begin{array}{ccccc}
    * & * & * & \cdots & * \\
    M_{2,1} & * & * & \cdots & * \\
    0 & M_{3,2} & * & \cdots & * \\
    \vdots & \ddots & \ddots & \ddots & \vdots \\
    0 & \cdots & 0 & M_{n,n-1} & *
  \end{array}\right) U^{-1}
=
\left(\begin{array}{ccccc}
    \triangle & \triangle & \triangle & \cdots & \triangle \\
    M_{2,1} & \triangle & \triangle & \cdots & \triangle \\
    0 & M_{3,2} & \triangle & \cdots & \triangle \\
    \vdots & \ddots & \ddots & \ddots & \vdots \\
    0 & \cdots & 0 & M_{n,n-1} & \triangle
\end{array}\right),
\]
where ``$*$" and ``$\triangle$" stand for unspecified entries. 
\begin{proof}[Proof of Proposition \ref{prop-lowpart}]
Consider the $(i,j)$-entry of $UMU^{-1}$:
\[
(UMU^{-1})_{i,j} = \sum_{1 \le l,m \le n} U_{i,l} M_{l,m} (U^{-1})_{m,j} = \sum_{1 \le l,m \le n} (-1)^{m+j} \binom{n-i}{n-l} \binom{n-m}{n-j} M_{l,m}.
\]
For fixed $i$ and $j$, we have $\binom{n-i}{n-l} \binom{n-m}{n-j} = 0$ unless $l \ge i$ and $m \le j$.

In the $i - j \ge 2$ case, the condition $l \ge i$ and $m \le j$ implies that $l - m \ge i - j \ge 2$. Hence $M_{l,m} = 0$. Therefore, $(UMU^{-1})_{i,j} = 0 = M_{i,j}$.

In the $i - j = 1$ case, one can check that
\[
(-1)^{m+i-1} \binom{n-i}{n-l} \binom{n-m}{n-i+1} M_{l,m} = \begin{cases}
                                                             M_{i,i-1}, & \mbox{if } l=i \text{ and } m=i-1; \\
                                                             0, & \mbox{otherwise}.
                                                           \end{cases}
\]
And hence $(UMU^{-1})_{i, i-1} = M_{i,i-1}$. This completes the proof.
\end{proof}

\section{A generalization of Theorem \ref{thm-main}}\label{sec-gen}

The observation $(\TC^{(n)})_{i,i} = (\TC^{(n)})_{i,i-1} + (\TC^{(n)})_{i,i+1}$ leads us to consider matrix $A$ satisfying $A_{i,i} = r + A_{i,i-1} + A_{i,i+1}$, where $r$ is a parameter. Further exploration gives Theorem \ref{thm-gen}.

\begin{thm}\label{thm-gen}
Let $n$ be a positive integer. Suppose $\{a_i^{(t)}\} (1 \le t \le n-1)$ and $\{b_i\}$ are sequences of real numbers satisfying that $b_i = 0$ if $i \le 1$ or $i > n$, and $a_i^{(t)} = 0$ if $i+t > n$ or $i \le 0$. Additionally, for any $0 \le k \le n-2$, there exists a constant number $\lambda(k)$ such that
\begin{equation}\label{e-condi}
\begin{aligned}
b_{i+1} - b_i & + \sum_{t=1}^{n-1} \sum_{l=0}^{t-1} (-1)^{l+t} \binom{t-l+k}{t-1} \binom{t-1}{l} a_{i+1-l+k}^{(t)} \\
& -  \sum_{t=1}^{n-1} \sum_{l=0}^{t-1} (-1)^{l+t} \binom{t-1-l+k}{t-1} \binom{t-1}{l} a_{i-l+k}^{(t)} = \lambda(k)
\end{aligned}
\end{equation}
for every $1 \le i \le n-k$. Let $U$ be the matrix given by \eqref{e-U}, and $A$ be the $n \times n$ matrix with entries
\[
A_{i,j} = \begin{cases}
            r + b_i + \sum_{t=1}^{n-1} (-1)^{t-1} a_i^{(t)},&\mbox{if } i=j;\\
            b_i,& \mbox{if } i=j+1; \\
            a_i^{(t)}, & \mbox{if } i=j-t \ (1 \le t \le n-1);\\
            0, &\mbox{otherwise},
          \end{cases}
\]
where $r$ is a parameter. Then
\[
(U A U^{-1})_{i,j} = \begin{cases}
            r + \sum_{k=0}^{n-i-1} \lambda(k), & \mbox{if }\ i = j; \\
            b_i, & \mbox{if }\ i = j + 1; \\
            0, & \mbox{otherwise}.
          \end{cases}
\]
In particular, if $a_i^{(t)} = 0$ for any $2 \le t \le n-1$, then
\[
(U A U^{-1})_{i,j} = \begin{cases}
            r + a_i^{(1)} + (n-i) (b_{i+1} - b_i), & \mbox{if }\ i = j; \\
            b_i, & \mbox{if }\ i = j + 1; \\
            0, & \mbox{otherwise}.
          \end{cases}
\]
\end{thm}

If we set $r=0$, $b_i = (n+1-i)(1-i)$, and $a_i^{(t)} = \begin{cases}
                                (n+1-i)(n-i), & \mbox{if } t=1; \\
                                0, & \mbox{if } 2 \le t \le n-1,
                              \end{cases}$
then Theorem \ref{thm-gen} reduces to Theorem \ref{thm-main}.

Given a positive integer $n$, let
\[
f_j := \begin{cases}
         e_1, & \mbox{if } j=1; \\
         e_{j-1} + e_j, & \mbox{if } 2 \le j \le n
       \end{cases}
\qquad \text{and} \qquad
g_j := \sum_{i=1}^{j} (-1)^{i+j} e_i \ (1 \le j \le n),
\]
where $e_i$ is the standard $i$th unit vector of dimension $n$ for each $1 \le i \le n$. Let $P$ and $Q$ be the $n \times n$ matrices defined by
\begin{equation}\label{e-PQ}
P := (f_1, f_2, \dots, f_n) 
\qquad \text{and} \qquad
Q := (g_1, g_2, \dots, g_n).
\end{equation}
One can check that
\begin{equation}\label{e-PQ-1}
P = Q^{-1}.
\end{equation}

\begin{lem}\label{lem-PU1}
Let $P$ and $U$ be the matrices with entries as in \eqref{e-PQ} and \eqref{e-U}, respectively. Then we have
\[
\left(\begin{array}{cc}
        U_1 & O \\
        O & 1
      \end{array}\right) P = U,
\]
where $U_1$ is the $(n-1) \times (n-1)$ nonsingular upper triangular matrix with entries
\begin{equation}\label{e-U1}
(U_1)_{i,j} = \binom{n-1-i}{n-1-j}.
\end{equation}
\end{lem}
\begin{proof}
Let $D = \left(\begin{array}{cc}
U_1 & O \\
O & 1
\end{array}\right) P$. It is easy to check that
\[
D_{i,n} = U_{i,n} = 1,
\quad \text{ and } \quad
D_{n,j} = U_{n,j} = \begin{cases}
                        1, & \mbox{if } j = n; \\
                        0, & \mbox{if } j < n.
                    \end{cases}
\]
For $i, j \le n-1$, we have
\begin{align*}
D_{i,j} &= \sum_{l=1}^{n-1} (U_1)_{i,l} P_{l,j} = (U_1)_{i,j-1} P_{j-1,j} + (U_1)_{i,j} P_{j,j} \\
&= \binom{n-1-i}{n-j} + \binom{n-1-i}{n-1-j} = \binom{n-i}{n-j}  & \text{(by \eqref{e-comid1})}\\
&= U_{i,j}.
\end{align*}
The proof is completed.
\end{proof}

\begin{proof}[Proof of Theorem \ref{thm-gen}]
We prove the theorem by induction on $n$. It is trivial for $n = 1$ or $2$. Assume the theorem holds for $n-1$. We perform elementary operations on $A$ as follows:
\begin{enumerate}
  \item[Step 1.] (Column operations) Add column $1$ multiplied by $-1$ to column $2$, then column $2$ multiplied by $-1$ to column $3$, $\dots$,  column $n-1$ multiplied by $-1$ to column $n$.

  \item[Step 2.] (Row operations) Add row $2$ to row $1$, then row $3$ to row $2$, $\dots$, row $n$ to row $n-1$.
\end{enumerate}
These operations convert $A$ into
\[
\left(\begin{array}{cc}
  A_1 & O \\
  B & r
\end{array}\right),
\]
where $B = (0, \dots, 0, b_n) \in \RR^{1 \times (n-1)}$, and $A_1$ is the $(n-1) \times (n-1)$ matrix whose entries are given by
\[
(A_1)_{i,j} = \begin{cases}
                r + b_{i+1} + \sum_{t=1}^{n-1} (-1)^{t-1} a_i^{(t)}, & \mbox{if } i = j;   \\
                b_i, & \mbox{if } i = j+1;  \\
                c_i^{(s)}, & \mbox{if } i = j-s \ (1 \le s \le n-2); \\
                0, & \mbox{otherwise}
              \end{cases}
\]
with
\begin{equation}\label{e-ci}
c_i^{(s)} = (-1)^{s} \Big( \sum_{t=s+1}^{n-1} (-1)^{t-1} a_i^{(t)} - \sum_{t=s}^{n-1} (-1)^{t-1} a_{i+1}^{(t)} \Big).
\end{equation}
Formula \eqref{e-ci} holds for $1 \le i \le n-1$, and it is easy to see that $c_i^{(s)} = 0$ if $s+i>n-1$. We further set $c_i^{(s)} = 0$ for $i \le 0$. Performing the above elementary operations on $A$ is equivalent to computing the matrix product $PAQ$. Thus, we have
\begin{equation}\label{e-PAQ}
P A Q = \left(\begin{array}{cc}
          A_1 & O \\
          B & r
        \end{array}\right).
\end{equation}
By \eqref{e-condi},
\begin{align*}
& \quad \lambda(0) - b_{i+1} + b_i + \sum_{s=1}^{n-2} (-1)^{s-1} c_i^{(s)} \\
&= \sum_{t=1}^{n-1} \sum_{l=0}^{t-1} (-1)^{l+t} \binom{t-l}{t-1} \binom{t-1}{l} a_{i+1-l}^{(t)} -  \sum_{t=1}^{n-1} \sum_{l=0}^{t-1} (-1)^{l+t} \binom{t-1-l}{t-1} \binom{t-1}{l} a_{i-l}^{(t)} \\
& \quad - \sum_{s=1}^{n-2} \Big( \sum_{t=s+1}^{n-1} (-1)^{t-1} a_i^{(t)} - \sum_{t=s}^{n-1} (-1)^{t-1} a_{i+1}^{(t)} \Big) \\
&= \sum_{t=1}^{n-2} (-1)^t t a_{i+1}^{(t)} + \sum_{t=1}^{n-1} (-1)^{t-1} (t - 1) a_{i}^{(t)} - \sum_{t=1}^{n-1} (-1)^t a_{i}^{(t)} - \sum_{t=2}^{n-1} \sum_{s=1}^{t-1} (-1)^{t-1} a_i^{(t)} \\
& \quad + \sum_{t=1}^{n-2} \sum_{s=1}^{t} (-1)^{t-1} a_{i+1}^{(t)}\\
&= \sum_{t=1}^{n-2} (-1)^t t a_{i+1}^{(t)} + \sum_{t=1}^{n-1} (-1)^{t-1} (t - 1) a_{i}^{(t)} - \sum_{t=1}^{n-1} (-1)^t a_{i}^{(t)} - \sum_{t=2}^{n-1} (-1)^{t-1} (t-1) a_i^{(t)} \\
& \quad + \sum_{t=1}^{n-2} (-1)^{t-1} t a_{i+1}^{(t)} \\
&= \sum_{t=1}^{n-1} (-1)^{t-1} a_i^{(t)},
\end{align*}
where the second equation holds by $a_{i+1}^{(n-1)} = 0$ since $i+n>n$, $\binom{t-l}{t-1} = 0$ if $l > 1$, and $\binom{t-1-l}{t-1} = 0$ if $l > 0$. Therefore, the entries of $A_1$ can be rewritten as
\[
(A_1)_{i,j} = \begin{cases}
                (r + \lambda(0))+ b_i + \sum_{s=1}^{n-2} (-1)^{s-1} c_i^{(s)}, & \mbox{if } i = j;  \\
                b_i, & \mbox{if } i = j+1; \\
                c_i^{(s)}, & \mbox{if } i = j-s \ (1 \le s \le n-2);   \\
                0, & \mbox{otherwise.}
              \end{cases}
\]
We claim that
\begin{equation}\label{e-condi-c1}
\begin{aligned}
& \quad \sum_{s=1}^{n-2} \sum_{l=0}^{s-1} (-1)^{l+s} \binom{s-1-l+k}{s-1} \binom{s-1}{l} c_{i-l+k}^{(s)} \\
&= \sum_{t=1}^{n-1} \sum_{l=0}^{t-1} (-1)^{l+t} \binom{t-l+k}{t-1} \binom{t-1}{l} a_{i+1-l+k}^{(t)}.
\end{aligned}
\end{equation}
Direct substitution yields
\begin{align*}
& \quad \sum_{s=1}^{n-2} \sum_{l=0}^{s-1} (-1)^{l+s} \binom{s-1-l+k}{s-1} \binom{s-1}{l} c_{i-l+k}^{(s)} \\
&= \sum_{s=1}^{n-2} \sum_{l=0}^{s-1} (-1)^{l} \binom{s-1-l+k}{k} \binom{k}{l} \Big( \sum_{t=s+1}^{n-1} (-1)^{t-1} a_{i-l+k}^{(t)} - \sum_{t=s}^{n-1} (-1)^{t-1} a_{i+1-l+k}^{(t)} \Big),
\end{align*}
where the equation holds by \eqref{e-comid3}. The computation is naturally divided into two parts, denoted $S_1$ and $S_2$, respectively. To compute $S_1$ and $S_2$, we need another well-known combinatorial identity (the proof is omitted):
\begin{equation}\label{e-comid4}
\binom{n+1}{k+1} = \sum_{i=k}^{n} \binom{i}{k}.
\end{equation}
Then
\begin{align*}
S_1 &= \sum_{s=1}^{n-2} \sum_{l=0}^{s-1} \sum_{t=s+1}^{n-1} (-1)^{l+t-1} \binom{s-1-l+k}{k} \binom{k}{l} a_{i-l+k}^{(t)} \\
&= \sum_{t=2}^{n-1} \sum_{l=0}^{t-2} \sum_{s=l+1}^{t-1} (-1)^{l+t-1} \binom{s-1-l+k}{k} \binom{k}{l} a_{i-l+k}^{(t)} \\
&= \sum_{t=2}^{n-1} \sum_{l=0}^{t-2} (-1)^{l+t-1} \binom{t-1-l+k}{k+1} \binom{k}{l} a_{i-l+k}^{(t)} & \text{(by \eqref{e-comid4})} \\
&= \sum_{t=2}^{n-1} \sum_{l=1}^{t-1} (-1)^{l+t} \binom{t-l+k}{k+1} \binom{k}{l-1} a_{i+1-l+k}^{(t)} \\
&= \sum_{t=1}^{n-1} \sum_{l=0}^{t-1} (-1)^{l+t} \binom{t-l+k}{k+1} \binom{k}{l-1} a_{i+1-l+k}^{(t)},
\end{align*}
where the last equation holds by $\binom{k}{l-1} = 0$ if $l=0$. Observe that $a_{i+1-l+k}^{(n-1)} = 0$ unless $i+1-l+k=1$, i.e., $k = l -i < l$, which implies that $\binom{k}{l} = 0$. This leads the fact
\begin{equation}\label{e-n-1=0}
\binom{k}{l} a_{i+1-l+k}^{(n-1)} = 0.
\end{equation}
Then
\begin{align*}
S_2 &= \sum_{s=1}^{n-2} \sum_{l=0}^{s-1} \sum_{t=s}^{n-1} (-1)^{l+t} \binom{s-1-l+k}{k} \binom{k}{l} a_{i+1-l+k}^{(t)} \\
&= \sum_{s=1}^{n-2} \sum_{l=0}^{s-1} \sum_{t=s}^{n-2} (-1)^{l+t} \binom{s-1-l+k}{k} \binom{k}{l} a_{i+1-l+k}^{(t)} & (\text{by \eqref{e-n-1=0}})\\
&= \sum_{t=1}^{n-2} \sum_{l=0}^{t-1} \sum_{s=l+1}^{t} (-1)^{l+t} \binom{s-1-l+k}{k} \binom{k}{l} a_{i+1-l+k}^{(t)} \\
&= \sum_{t=1}^{n-2} \sum_{l=0}^{t-1} (-1)^{l+t} \binom{t-l+k}{k+1} \binom{k}{l} a_{i+1-l+k}^{(t)}  & \text{(by \eqref{e-comid4})} \\
&= \sum_{t=1}^{n-1} \sum_{l=0}^{t-1} (-1)^{l+t} \binom{t-l+k}{k+1} \binom{k}{l} a_{i+1-l+k}^{(t)}. & (\text{by \eqref{e-n-1=0}})
\end{align*}
Thus,
\begin{align*}
S_1 + S_2 &= \sum_{t=1}^{n-1} \sum_{l=0}^{t-1} (-1)^{l+t} \binom{t-l+k}{k+1} \Big( \binom{k}{l} + \binom{k}{l-1} \Big) a_{i+1-l+k}^{(t)} \\
&= \sum_{t=1}^{n-1} \sum_{l=0}^{t-1} (-1)^{l+t} \binom{t-l+k}{k+1} \binom{k+1}{l} a_{i+1-l+k}^{(t)} & (\text{by \eqref{e-comid1}})\\
&= \sum_{t=1}^{n-1} \sum_{l=0}^{t-1} (-1)^{l+t} \binom{t-l+k}{t-1} \binom{t-1}{l} a_{i+1-l+k}^{(t)}. & (\text{by \eqref{e-comid3}})
\end{align*}
That is, \eqref{e-condi-c1} holds. Replacing $k$ by $k+1$ in \eqref{e-condi-c1} yields
\begin{equation}\label{e-condi-c2}
\begin{aligned}
& \quad \sum_{s=1}^{n-2} \sum_{l=0}^{s-1} (-1)^{l+s} \binom{s-l+k}{s-1} \binom{s-1}{l} c_{i+1-l+k}^{(s)} \\
&= \sum_{t=1}^{n-1} \sum_{l=0}^{t-1} (-1)^{l+t} \binom{t-l+k+1}{t-1} \binom{t-1}{l} a_{i+2-l+k}^{(t)}.
\end{aligned}
\end{equation}
By \eqref{e-condi}, \eqref{e-condi-c1} and \eqref{e-condi-c2}, we obtain that the entries of $A_1$ fits the condition:
\begin{align*}
b_{i+1} - b_i & + \sum_{s=1}^{n-2} \sum_{l=0}^{s-1} (-1)^{l+s} \binom{s-l+k}{s-1} \binom{s-1}{l} c_{i+1-l+k}^{(s)} \\
& -  \sum_{s=1}^{n-2} \sum_{l=0}^{s-1} (-1)^{l+s} \binom{s-1-l+k}{s-1} \binom{s-1}{l} c_{i-l+k}^{(s)}  = \lambda(k+1).
\end{align*}
By the induction hypothesis, we have
\[
(U_1 A_1 (U_1)^{-1})_{i,j} = \begin{cases}
            r + \lambda(0) + \sum_{k=0}^{(n-1)-i-1} \lambda(k+1) = r + \sum_{k=0}^{n-i-1} \lambda(k), & \mbox{if }\ i = j; \\
            b_i, & \mbox{if }\ i = j + 1; \\
            0, & \mbox{otherwise},
          \end{cases}
\]
where $U_1$ is the matrix with entries as in \eqref{e-U1}. Combining with Lemma \ref{lem-PU1} and Equation \eqref{e-PQ-1}, we have
\begin{align*}
U A U^{-1} &= \left(\begin{array}{cc}
        U_1 & O \\
        O & 1
      \end{array}\right)
      P A Q
      \left(\begin{array}{cc}
        (U_1)^{-1} & O \\
        O & 1
      \end{array}\right) \\
       &= \left(\begin{array}{cc}
        U_1 & O \\
        O & 1
      \end{array}\right)
      \left(\begin{array}{cc}
        A_1 & O \\
        B & r
      \end{array}\right)
      \left(\begin{array}{cc}
        (U_1)^{-1} & O \\
        O & 1
      \end{array}\right) & (\text{by \eqref{e-PAQ}})\\
      &= \left(\begin{array}{cc}
        U_1 A_1 (U_1)^{-1} & O \\
        B (U_1)^{-1} & r
      \end{array}\right)
      = \left(\begin{array}{cc}
        U_1 A_1 (U_1)^{-1} & O \\
        B & r
      \end{array}\right),
\end{align*}
where the last equation holds by the fact $B U_1 = B$. This completes the proof.
\end{proof}

\medskip
\noindent \textbf{Acknowledgments:}
The authors would like to thank Professors Guoce Xin (Capital Normal University) and Arthur L. B. Yang (Nankai University) for helpful discussions. This work was supported by the Postdoctoral Fellowship Program and China Postdoctoral Science Foundation (No. BX20250066).

\end{document}